\documentclass[a4paper, 11pt]{amsart}
\input xy   
\xyoption{all}
\swapnumbers
\usepackage{amssymb}  
\usepackage{amsthm}
\usepackage{enumerate}
\numberwithin{equation}{section}
\theoremstyle{plain}
  \begingroup
        \newtheorem{theorem}[equation]{Theorem}
        \newtheorem{lemma}[equation]{Lemma}
        \newtheorem{proposition}[equation]{Proposition}

	\newtheorem{definition}[equation]{Definition}

        \newtheorem{sinnadaitalica}[equation]{}
\endgroup

\theoremstyle{definition}
  \begingroup

\endgroup

\newcommand{\mr}[1]{\buildrel {#1} \over \longrightarrow}

\newcommand{\Mr}[1]{\buildrel {#1} \over \Longrightarrow}

\newcommand{\pmr}[2]
{
\xymatrix@C=5ex@R=2.4ex
         {
          {} \ar@<1.6ex>[r]^{#1} 
	     \ar@<-1.1ex>[r]^{#2} & {}
         }
}

\newcommand{\pml}[2]
{
\xymatrix@C=5ex@R=2.4ex
         {
            {} 
          & {} \ar@<1.0ex>[l]_{#1} 
	       \ar@<-1.7ex>[l]_{#2}
         }
}

\newcommand{\cellr}[3]
{
\xymatrix@C=7ex@R=2.4ex
         {
          {} \ar@<1.6ex>[r]^{#1} 
          \ar@{}@<-1.3ex>[r]^{\!\! #2 \, \!\Downarrow}
                                         \ar@<-1.1ex>[r]_{#3} & {}
         }
}

\newcommand{\celll}[3]
{
\xymatrix@C=7ex@R=2.4ex
         {
            {} 
          & {} \ar@<1.0ex>[l]^{#1} 
          \ar@{}@<-1.7ex>[l]^{\!\! #2 \, \!\Downarrow}
	                                 \ar@<-1.7ex>[l]_{#3}
         }
}

\newcommand{\Colim}[1]{\underrightarrow{\cc{L}im}{\; #1}}

\newcommand{\cc}{\mathcal}

\usepackage{colordvi}

\newcommand{\cqd}{\hfill$\Box$}

\theoremstyle{definition}
  \begingroup

\endgroup

\def\xto#1{\xrightarrow{#1}}

		\def\cat{\mathcal C at}

		\def\sit{\mathcal S it}

		\def\a{\alpha}				
                
		\def\g{\gamma}

		\def\D{\mathcal D}
                \def\l{\lambda}

		\def\e{\epsilon}
		\def\E{\mathcal E}
		\def\F{\mathcal F}
		\def\A{\mathcal A}
		\def\B{\mathcal B}
		
				\def\C{\mathcal C}

\begin{document}



\title{A construction of 2-cofiltered bilimits of topoi}

\author{Eduardo J. Dubuc, \; Sergio Yuhjtman}


 
\maketitle


{\sc introduction}

We show the existence of bilimits of 2-cofiltered diagrams of
topoi, generalizing the construction of cofiltered bilimits developed in 
 \cite{G2}. For any given such diagram, we show that it can be
represented by a 2-cofiltered diagram of small sites with finite
limits, and we construct a small site for the inverse limit
topos. This is done by taking the 2-filtered bicolimit of the
underlying categories and inverse image functors. We use the
construction of this bicolimit developed in \cite{DS}, where it is
proved that if the categories in the diagram have finite limits and
the transition functors are exact, then the bicolimit category has
finite limits and the pseudocone functors are exact. An application of our result here is the fact that every Galois topos  \mbox{has points \cite{D}.} 

\section{Background, terminology and notation} \label{background}

In this section we recall some $2$-category and topos 
theory that we shall explicitly need, and in this way fix notation 
and terminology. We also include some in-edit proofs when it seems necessary. We distinguish between \emph{small} and \emph{large}
sets. Categories are supposed to have small hom-sets. A category with
large hom-sets is called \emph{illegitimate}. 
 
\vspace{1ex}

{{\bf Bicolimits}}

By a \emph{2-category} we mean a $\cc{C}at$ enriched category, and
\mbox{\emph{2-functors}} 
are $\cc{C}at$ functors, where $\cc{C}at$ is the category of small categories. Given a 2-category, as usual, we denote
horizontal composition by juxtaposition, and vertical composition  
by a $''\circ''$.  We consider juxtaposition more binding than
$''\circ''$ (thus $xy\circ z$ means $(xy)\circ z$). If $\cc{A},\;
\cc{B}$ are $2$-categories ($\cc{A}$ small), we will 
denote by $[[\cc{A}, \cc{B}]]$ the $2$-category which has as objects
the $2$-functors, as arrows the  \emph{pseudonatural transformations},
and as $2$-cells the \emph{modifications} (see \cite{G}
I,2.4.). Given $F,\,G,\,H\,: \cc{A} \mr{} \cc{B}$,
there is a functor:
\begin{equation} \label{3cat}
[[\cc{A},\, \cc{B}]](G,\,H) \times [[\cc{A},\, \cc{B}]](F,\,G)
\mr{} [[\cc{A},\, \cc{B}]](F,\,H)
\end{equation} 
To have a handy reference
we will explicitly describe these data in the particular cases
we use. 

A \emph{pseudocone} of a diagram given by a 2-functor $\A \xto F \B$
to an object \mbox{$X \in \cc{B}$} is a pseudonatural 
transformation \mbox{$F \mr{h} X$} from $F$ to the 2-functor
which is constant at $X$. It consists of a family of 
arrows \mbox{$(h_A: FA \to X)_{A \, \in \cc{A}}$,} and a family
of invertible $2$-cells \mbox{$(h_u: h_A \to h_B \circ Fu)_{(A \mr{u}
    B) \, \in \cc{A}}$.} 
A morphism $g \Mr{\varphi} h$ of pseudocones (with same vertex) 
is a modification,  as such, it consists of a
family of $2$-cells 
$(g_A \Mr{\varphi_A} h_A)_{A \in \cc{A}}$. These data is subject to
the following:

\begin{sinnadaitalica} [{\it Pseudocone and morphism of pseudocone
    equations}] \label{PCequations} ${}$

\vspace{1ex}

pc0. $h_{id_A} = id_{h_A}$, \hfill    for each object $A$ \hspace{3ex}

pc1.  $h_v Fu \circ h_u  =  h_{vu}$, \hfill for each pair
of arrows \; $A \mr{u} B \mr{v} C$ \hspace{3ex}
                                     
pc2.  $h_B F\g \circ h_v = h_u$, \hfill for each 2-cell \;
$A \cellr{u}{\gamma}{v} B$  \hspace{3ex} 

pcM. $h_u \circ \varphi_A = \varphi_B Fu \circ g_u$, \hfill for each
arrow \; $A \mr{u} B$ \hspace{3ex}

\end{sinnadaitalica}
We state and prove now a lemma which, although expected, needs
nevertheless a proof, and for which we do not have a reference in the
literature. As the reader will realize, the statement concerns general
pseudonatural transformations, but we treat here the particular case
of pseudocones.
\begin{lemma} \label{translacion}
	Let $\A \xto F \B$ be a 2-functor and $F \mr{g} X$ a
        pseudocone. Let $FA \mr{h_A} X$ be a family of morphisms
        together with 
        invertible $2$-cells $g_A \Mr{\varphi_A} h_A$.
   Then, conjugating by $\varphi$ determines a pseudocone structure
for $h$, unique such that $\varphi$ becomes an isomorphism of
pseudocones. 
\end{lemma}
\begin{proof}
If $\varphi$ is to become a pseudocone morphism, the equation pcM.
\mbox{$\varphi_B Fu \,\circ\, g_u = h_u \circ \varphi_A$} must hold.
Thus, $h_u = \varphi_B Fu \circ g_u \circ \varphi_A^{-1}$
determines and \mbox{defines $h$.}  The pseudocone equations
\ref{PCequations} for $h$ follow from the respective equations
for $g$: 

	pc0.
$h_{id_A}\;=\;\varphi_A \circ g_{id_A} \circ \varphi_A^{-1} \;=\;
 \varphi_A \circ id_{g_A} \circ \varphi_A^{-1} \;=\; id_{h_A} $ 
	
	\vspace{5pt}

	pc1. $A \xto u B \xto v C$:

$h_v Fu \circ h_u \;\; = $ \hfill 
$(\varphi_C Fv \circ g_v \circ \varphi_B^{-1})Fu \circ \varphi_B Fu
  \circ g_u \circ \varphi_A^{-1} \;\;\;=$  \hspace{9ex}
 
\hfill $\varphi_C F(vu) \circ g_v Fu \circ \varphi_B^{-1} Fu
\circ \varphi_B Fu \circ g_u \circ \varphi_A^{-1} \;\;\;=$ 
\hspace{9ex}

\hfill $\varphi_C F(vu) \circ g_v Fu \circ g_u \circ\varphi_A^{-1}
\;\;\;=$  \hspace{9ex}

\hfill $\varphi_C F(vu) \circ g_{vu} \circ \varphi_A^{-1} \;\;\;=$
\hspace{2ex} $h_{vu}$ \hspace{2ex} 
	 
	pc2. For $\xymatrix{A \ar@<1ex>[r]^u \ar@{}[r]|{\Uparrow \g}
\ar@<-1ex>[r]_v & B}$ we must see 
	$h_B F\g \circ h_v = h_u$. This is the same as
	$h_B F\g \circ \varphi_B Fv \circ g_v \circ
        \varphi_A^{-1}=\varphi_B Fu \circ g_u
        \circ\varphi_A^{-1}$. 
	Canceling $\varphi_A^{-1}$ and composing with 
        $(\varphi_B Fu)^{-1}$ yields  
	(1) $(\varphi_B Fu)^{-1} \circ h_B F\g \circ \varphi_B Fv
        \circ g_v = g_u$. 
From the compatibility between vertical and horizontal
composition it \mbox{follows} \mbox{$(\varphi_B Fu)^{-1} \circ h_BF\g \circ \varphi_B Fv \;=$}  
$(\varphi_B^{-1} \circ h_B \circ \varphi_B)(Fu \circ F\g \circ Fv) = g_B
F\g$. Thus, after replacing,
(1) becomes $g_B F\g \circ g_v = g_u$. 
\end{proof}

Given a small 2-diagram $\cc{A} \mr{F} \cc{B}$, the category of pseudocones and its 
morphisms is, by definition,  \mbox{$pc \cc{B}(F, X) =
  [[\cc{A},\, 
      \cc{B}]](F,\, X)$.} Given a pseudocone $F \mr{f} Z$ and a
$2$-cell 
$Z \cellr{s}{\xi}{t} X$, it is clear and straightforward how to define
a morphism of pseudocones $F \cellr{sf}{\xi f}{tf} X$  which is the
composite  
$F \mr{f} Z \cellr{s}{\xi}{t} X$. This is a particular case of
\ref{3cat}, thus composing with $f$  
determines a functor (denoted $\rho_f$)          
\mbox{$\cc{B}(Z,\, X)  \mr{\rho_f} pc \cc{B}(F,\, X)$.} 

%

\begin{definition} \label{bicolimit}
A pseudocone $F \mr{\lambda} L$ is a \emph{bicolimit} of $F$ if for
every object $X \in \cc{B}$, the
functor $\cc{B}(L,\, X) \mr{\rho_\lambda} pc\cc{B}(F,\, X)$ is
an  
equivalence of categories. This amounts to the following:            

{\it bl}) Given any pseudocone $F \mr{h} X$, there exists an arrow
$L \xto{\ell} X$ and an invertible morphism of pseudocones 
$\;h \Mr{\theta} \ell \lambda$. Furthermore, given any other 
$L \xto{t} X$  and  $\;h \Mr{\varphi} t \lambda$,   
there exists a \emph{unique}
$2$-cell $\;\ell \Mr{\xi} t$ such that  
$\varphi = (\xi \lambda) \circ \theta$ (if $\varphi$ is invertible,
then so it is $\xi$).

\end{definition}
\begin{definition}
When the functor $\cc{B}(L,\, X) \mr{\rho_\lambda} pc\cc{B}(F,\,
  X)$ 
is an isomorphism of categories, the bicolimit is said to be a 
\emph{pseudocolimit}.
\end{definition}
%
%

It is known that the $2$-category $\cat$ of small categories has all
 small pseudocolimits, then a ``fortiori'' all small
  bicolimits (see for example \cite{S}). Given a 2-functor $\cc{A}
  \mr{F} \cc{C}at$ we denote by $\Colim{F}$ the vertex of a bicolimit
  cone. 

In \cite{DS} a special
  construction 
  of the pseudocolimit of a 2-filtered  diagram of categories
   (not necessarily small) is made, and using this construction it is
  proved a result 
  (\mbox{theorem \ref{key} below}) which is the key to our
  construction of small
  $2$-filtered bilimits of topoi. Notice that even if the categories
  of the system are large, condition {\it bl)} in definition
  \ref{bicolimit} makes sense and it defines the bicolimit of large
  categories.

We denote by $\cc{CAT}_{fl}$ the \emph{illegitimate} (in the sense
  that its hom-sets are large) 2-category of
finitely complete 
categories and exact (that is, finite limit preserving) functors.

 \begin{theorem}[\cite{DS} Theorem 2.5] \label{key}
  $\cc{CAT}_{fl} \subset
 \cc{CAT}$ is closed under 2-filtered 
  pseudocolimits. Namely, given any 2-filtered diagram $\A \xto F
  \cc{CAT}_{fl}$, the
 pseudocolimit pseudocone $FA \mr{\lambda_A}
 \Colim{F}$ taken in $\cc{CAT}$ is a
 pseudocolimit cone in $\cc{CAT}_{fl}$. If the index 2-category
 $\cc{A}$ as well as all the categories $FA$ are small, then
 $\Colim{F}$ is a small category.
 \cqd
\end{theorem}

{\bf Topoi}


By a \emph{site} we mean a category furnished with a (Grothendieck)
topology, and a small set of objects capable of covering any
object (called \emph{topological generators} in \cite{G1}). \emph{To
  simplify we will consider  only sites with finite 
  limits.}    
A \emph{morphism} of sites with finite limits $\D \mr{f} \C$ is a
\emph{continous} (that is, cover preserving) and exact functor in the
other direction $\C \mr{f^*} \D$.
%
%
A $2$-cell $\cc{D} \cellr{f}{\gamma}{g} \cc{C}$ is a natural
 transformation $\cc{C} \cellr{g^*}{\gamma}{f^*}
 \cc{D}$ \footnote{Notice that $2$-cells are also taken in the
   opposite 
   direction. This is Grothendieck original convention, later changed
   by some authors.}. Under the presence of topological generators it
 can be easily seen there is only a small set of
 natural 
  transformations between any two continous functors. We denote by
  $\sit$ the resulting 2-category of sites with 
 finite 
limits.  We denote by $\cc{S}it^*$ the $2$-category whose objects
are the sites, but taking as arrows and $2$-cells the functors $f^*$
and natural transformations respectively. Thus $\cc{S}it$ is obtained
by formally inverting the arrows and the $2$-cells of $\cc{S}it^*$. We
have by definition 
$\cc{S}it(\cc{D},\cc{C}) = \cc{S}it^*(\cc{C},\cc{D})^{op}$.

\vspace{1ex}

A topos (also ``Grothendieck topos'') is a category equivalent
to the category of sheaves on a site. Topoi are
considered as 
sites furnishing them with the canonical topology. This determines a
full subcategory \mbox{$\cc{T}op^* \subset \cc{S}it^*$,}
$\cc{T}op^*(\cc{F},\,\cc{E})\;=\; \cc{S}it^*(\cc{F},\,\cc{E})$.

 A morphism of topoi (also ``geometric morphism'') $\E \xto f \F$ is a
 pair of adjoint functors ${f^* \dashv f_*}$  
(called inverse and direct image respectively) 
 $\xymatrix{\E \ar@<.5ex>[r]^{f_*} & \F \ar@<.5ex>[l]^{f^*}}$
 together with an adjunction isomorphism $[f^*C,D] \xto \cong
 [C,f_*D]$. Furthermore, $f^*$ is required to preserve finite
 limits.   
Let $\cc{T}op$ be the \mbox{2-category} of topos with geometric
morphisms. 2-arrows are pairs  of natural transformations ($f^*
\Rightarrow g^*$,   
$g_* \Rightarrow f_*$) compatible with the adjunction (one of the
natural transformations completely determines the other). 
The inverse image  $f^*$ of a morphism is an arrow in $\cc{T}op^*
\subset \cc{S}it^*$. This determines a forgetful 2-functor (identity
on the objects) 
$\cc{T}op \mr{} \cc{S}it$ which establish an equivalence of categories 
\mbox{$\cc{T}op(\cc{E},\,\cc{F}) \;\cong\;
  \cc{S}it(\cc{E},\,\cc{F})$}. Notice that    
$\cc{T}op(\cc{E},\cc{F}) \cong \cc{T}op^*(\cc{F},\cc{E})^{op}$, not an
equality. 

\vspace{1ex}

We recall a basic result in the theory of morphisms of Grothendieck
topoi  \cite{G1} expose IV, 4.9.4. (see for example \cite{MM} Chapter VII, section 7).
\begin{lemma} \label{Diaconescu}
	Let $\C$ be a site with finite limits, and 
$\C \xto {\e^*} \widetilde{\C}$ the canonical morphism of sites to  
	the topos of sheaves $\widetilde{\C}$. Then
        for  any topos $\F$, composing
        with $\e^*$ determines a functor
	$\cc{T}op^*(\widetilde{\C},\, \F) \xto {\cong}
        \cc{S}it^*(\C,\, \F)$ 
        which is an equivalence of categories. Thus, 
	$\cc{T}op(\cc{F},\, \widetilde{\C}) \xto {\cong}
        \cc{S}it(\F,\, \C)$. 
\end{lemma}
By the comparison lemma \cite{G1} Ex. III 4.1 we can state
it in the following form, to be used in the proof of lemma 
\ref{pseudodiaconescu}. 
\begin{lemma} \label{Diaconescubis}
	Let $\E$ be any topos and $\cc{C}$ any small set of generators
        closed under finite limits (considered as a site with the
        canonical topology). Then, 
        for  any topos $\F$, the
inclusion $\cc{C} \subset \cc{E}$ induce a restriction
functor 
\mbox{$\cc{T}op^*(\E,\F) \mr{\rho} \cc{S}it^*(\C,\F)$} which
is an equivalence of categories.  
\end{lemma}

\section{2-cofiltered bilimits of topoi}

Our work with sites is auxiliary to prove our results for topoi, and
for this all we need are sites with finite limits. The 2-category
$\sit$ has all small \mbox{2-cofiltered} pseudolimits, which 
are obtained by furnishing the 2-filtered pseudocolimit in
$\cc{CAT}_{fl}$ (\ref{key}) of the
underlying categories with the coarsest topology making the cone
injections site morphisms. Explicitly:  
\begin{theorem}\label{sitelimit}
Let $\cc{A}$ be a small 2-filtered 2-category, and \mbox{$\cc{A}^{op} \mr{F}
 \cc{S}it$} \mbox{($\cc{A} \mr{F} \cc{S}it^*$)} a \mbox{2-functor.}
 Then, the 
   category $\Colim{F}$ is  
furnished with a topology such that the pseudocone functors $FA
\mr{\lambda_A^*} \Colim{F}$ become continuous and induce an isomorphism
of categories 
\mbox{$\cc{S}it^*[\Colim{F},\, \cc{X}] \mr{\rho_{\lambda}}
\cc{P}\cc{C}{\cc{S}it^*}[F,\, \cc{X}]$.} The corresponding site 
is then a 
pseudocolimit of $F$ in the 2-category $\cc{S}it^*$. If each $FA$ is a
small category, then so it is $\Colim{F}$.
%
%
\end{theorem} 
\begin{proof}
Let $FA \mr{\lambda_A} \Colim{F}$ be the colimit pseudocone in
$\cc{CAT}_{fl}$.  We give $\Colim{F}$ the topology generated by the
families  
$\lambda_A c_\alpha \mr{} \lambda_A c$, where $c_\a \mr{} c$  
is a covering in some $FA$, $A \in \A$. With this topology, the
functors $\lambda_A$ become continuous, thus they correspond to  
site morphisms. This determines the upper horizontal arrow in the
following diagram (where the vertical arrows are full subcategories
and the lower horizontal arrow is an isomorphism):
$$
\xymatrix
         {
            \cc{S}it[\Colim{F},\, \cc{X}]
            \ar[r]  \ar[d]
          & pc\cc{S}it[F,\, \cc{X}]
            \ar[d]
         \\
            \cc{C}at_{fl}[\Colim{F},\, \cc{X}]
            \ar[r]^\cong
          & pc\cc{C}at_{fl}[F,\, \cc{X}]
         }
$$ 
To show that the upper horizontal arrow is an isomorphism we have to
check that given a pseudocone $h \in pc\cc{S}it[F,\, \cc{X}]$, the
unique functor \mbox{$f \in \cc{C}at_{fl}[\Colim{F},\, \cc{X}]$},
corresponding to 
$h$ under the lower arrow, is continuous. But this is clear
since from the equation $f\lambda = h$ it follows that it preserves
the generating covers, and thus all covers as well. Finally, by the
construction of $\Colim{F}$ in \cite{DS} we know that every object in
$\Colim{F}$ is of the form  $\lambda_A c$ for some $A \in \cc{A}$, 
$c \in FA$. It follows then that the collection of objects of the form
$\lambda_A c$, with $c$ varying on the set of topological generators
of each $FA$, is a
set of topological generators for  $\Colim{F}$.
\end{proof}

In the next proposition we show that any 2-diagram of topoi
restricts to a 2-diagram of small sites with finite limits by means of a 2-natural (thus a
fortiori pseudonatural) transformation. 
\begin{proposition}\label{res}
Given a 2-functor $\cc{A}^{op} \mr{\cc{E}} \cc{T}op$ there exists a 
2-functor $\cc{A}^{op} \mr{\cc{C}} \cc{S}it$ such that:

i) For any $A\in \A$, $\C_A$ is a \emph{small} full generating
subcategory of 
$\E_A$ 
closed under finite limits, considered as a site with the canonical
topology. 

ii) The
arrows and the $2$-cells in the 
$\cc{C}$ diagram are the restrictions of those in the $\cc{E}$
diagram: For any $2$ cell $A \cellr{u}{\gamma}{v} B$ in $\cc{A}$, the
following diagram commutes (where we omit notation for the action of
the $2$ functors on arrows and $2$-cells):
$$
\xymatrix@C=8ex
          {
              \cc{E}_A \ar@<1ex>[r]^{u^*} 
                 \ar@<-1.6ex>[r]_(0.5){v^*}^(0.5){\gamma \:\Downarrow}
                 & \cc{E}_B
            \\
              \cc{C}_A \ar@<1ex>[r]^{u^*} 
                 \ar@<-1.6ex>[r]_(0.5){v^*}^(0.5){\gamma \:\Downarrow}    \ar@{^(->}^{i_A}[u]
            & \cc{C}_B  \ar@{^(->}^{i_B}[u]
           }
$$
\end{proposition}
\begin{proof}
It is well known that any small set $\cc{C}$ of generators in a topos
can be enlarged so as to determine a (non canonical) small full subcategory 
$\overline{\cc{C}} \supset \cc{C}$ closed under finite
limits: Choose a limit cone for each finite diagram, and repeat
this in a denumarable process. On the other hand, for the validity of
condition ii) it is enough that for each transition functor $\cc{E}_A
\mr{u^*} \cc{E}_B$ and object $c \in \cc{C}_A$, we have $u^*(c) \in \cc{C}_B$
(with this, natural transformations restrict automatically).

Let's start with any set of generators $\cc{R}_A \subset \E_A$ for all
$A\in \cc{A}$. We
will naively add 
objects to these sets to remedy the failure of each condition
alternatively. In this way we achieve simultaneously the two
conditions:

Define $\cc{C}^0_A = \overline{\cc{R}}_A \supset
\cc{R}_A$.  
Define $\cc{R}_A^{n+1} =
\bigcup\limits_{X \xto u A} 
u^*(\cc{C}^n_X)$.  
	$\cc{R}_A^{n+1}$ is small because $\A$ is small.
	$\cc{C}^n_X \subset \cc{R}_A^{n+1}$ due to $id_A$.
	Suppose now $c\in \cc{R}_A^{n+1}$, $c=u^*(d)$ with $d\in
        \cc{C}^n_X$, and
        let $A
        \mr{v} B$ in $\cc{A}$.  
We have
$v^*(c)=v^*u^*(d) = (vu)^*(d)$,
	thus $v^*(c) \in \cc{R}_B^{n+1}$.
	Define $\C_A^{n+1} = \overline{\cc{R}_A^{n+1}} \supset
          \cc{R}_A^{n+1}$. 
	Then, it is straightforward to check that 
	$\C_A = \bigcup\limits_{n\in \mathbb{N}}\C_A^n$ satisfy the two
conditions.
\end{proof}

A generalization of lemma \ref{Diaconescubis} to pseudocones
holds.

\begin{lemma} \label{pseudodiaconescu}
Given any 2-diagram of topoi $\A^{op} \xto \E
\cc{T}op$, a restriction
$\A^{op} \xto \C \sit$ as before, and any topos $\cc{F}$, the
inclusions 
$\cc{C}_A \subset \cc{E}_A$ induce a restriction
functor 
\mbox{$pc\cc{T}op^*(\E,\F) \xto {\rho} pc\cc{S}it^*(\C,\F)$} which
is an equivalence of categories.
\end{lemma}
\begin{proof}
The restriction functor $\rho$ is just a particular case of \ref{3cat},
so it is well defined. We will check that it is essentially surjective
and fully-faithful. 
The
following diagram 
illustrates the situation:
$$
\xymatrix@R=1ex
         {
             {} & {} & {} & {} & {} & {}
          \\
            \cc{C}_A  \ar@{^(->}^{i_A}[r] \ar[dd]^{u^*} 
                      \ar@(u,u)[rrrd]^{g^*_A}
          & \cc{E}_A  \ar[dd]^{u^*} \ar[rrd]^{h_A^*}
                      \ar@{}[rrrru]^<{\hspace{5ex}\cong\, \varphi_A}
          \\ 
             {} & {} & {{}^{\hspace{-3ex}\Downarrow h_u}} & \cc{F}
          \\ 
            \cc{C}_B  \ar@{^(->}^{i_B}[r]
                      \ar@{}[uur]|\equiv
                      \ar@(d,d)[rrru]_{g^*_B}
          & \cc{E}_B  \ar[rru]_{h_B^*}
                      \ar@{}[rrrd]_<{\hspace{5ex}\cong\,\varphi_B}
          \\
            {} & {} & {} & {} & {}
        }
$$

\emph{essentially surjective}: 
Let \mbox{$g \in pc\cc{S}it^*(\C,\F)$.} For each $A \in
\A$, take by lemma \ref{Diaconescubis} $\E_A \xto {h^*_A} \F$,
        $\varphi_A$, 
\mbox{$h_A^* i_A \stackrel{\varphi_A}{\simeq}  g^*_A$}. 
By lemma \ref{translacion}, $h^*i$ inherits a pseudocone structure
such that $\varphi$ becomes a 
pseudocone isomorphism. 	
	For each arrow $A \xto u B$ we have $(h^*i)_A
\stackrel{(h^*i)_u}{\Rightarrow} (h^*i)_B u^*$.  Since $\rho_A$ is fully-faithful, there exists a unique $h_A^*
\stackrel{h_u}{\Rightarrow} h_B^* u^*$ extending $(h^*i)_u$. In this way we obtain 
data \mbox{$h^* = (h^*_A,\, h_u)$} that restricts to a
pseudocone. Again from 
the fully-faithfulness of each $\rho_A$ it is straightforward to check
that it satisfies the pseudocone equations \ref{PCequations}.

\vspace{1ex}

\emph{fully-faithful:}
Let \mbox{$h^*, l^* \in pc\cc{T}op^*(\E,\F)$} be two pseudocones, and
let  
$\widetilde{\eta}$ be a morphism between the pseudocones $h^*i$ and
$l^*i$. 
	We have natural transformations $\xymatrix{h^*_A i_A
\ar@{=>}[r]^{\widetilde{\eta_A}} & l^*_A i_A}$. 
	Since the inclusions $i_A$ are dense, we can extend
        $\widetilde{\eta_A}$  
uniquely to $\xymatrix{h^*_A \ar@{=>}[r]^{\eta_A} & l^*_A}$  such
that $\widetilde{\eta} = \eta \,i$. As before,  from 
the fully-faithfulness of each $\rho_A$ it is straightforward to check
that $\eta = (\eta_A)$  satisfies the morphism of pseudocone equation \ref{PCequations}. 
\end{proof}

\begin{theorem} \label{main}
	Let $\A^{op}$ be a small 2-filtered 2-category, and $\A^{op}
        \xto \E \cc{T}op$ be a 2-functor. Let  \mbox{$\A^{op}
\xto \C \cc{S}it$} be a restriction to small sites as in
\ref{res}. Then, the topos of sheaves $\widetilde{\Colim{\cc{C}}}$ on
the site $\Colim{\cc{C}}$ of \ref{sitelimit} is a bilimit of 
$\cc{E}$ in $\cc{T}op$, or, equivalently, a bicolimit in
$\cc{T}op^*$.  
\end{theorem}
\begin{proof}
	Let $\l^*$ be the pseudocolimit pseudocone $\C_A \xto {\l_A^*}
        \Colim{\cc{C}}$ in the \mbox{2-category $\cc{S}it^*$}
        (\ref{sitelimit}). Consider the composite pseudocone 
\mbox{$\C_A \xto {\l_A^*}  \Colim{\cc{C}} \mr{\varepsilon} 
                                    \widetilde{\Colim{\cc{C}}}$} and
let  $l^*$ be a
pseudocone from $\E$ to $\widetilde{\Colim{\cc{C}}}$ such that $l^*i \simeq
\e^*\l^*$ 
given by lemma \ref{pseudodiaconescu}. We have the following diagrams
commuting up to an isomorphism:
$$
\xymatrix
        {
          \cc{F} 
        & \widetilde{\Colim{\cc{C}}}
                                \ar[l]
                                \ar@{}@<-1.3ex>[rd]^{\cong}
        & \Colim{\cc{C}} \ar[l]_{\varepsilon^*}
        \\
          {} 
        & \cc{E} \ar[u]_{l^*} 
        & \cc{C} \ar[u]_{\lambda^*} \ar[l]_{i} 
        }
\hspace{5ex}         
\xymatrix
        {
          \cc{T}op^*(\widetilde{\Colim{\cc{C}}},\, \cc{F}) 
                                \ar[d]^{\rho_l}
                                \ar[r]^{\rho_\varepsilon}
                                \ar@{}@<-1.3ex>[rd]^{\cong}
        & \cc{S}it^*(\Colim{\cc{C}},\, \cc{F})
                                \ar[d]^{\rho_\lambda}
        \\
          pc\cc{T}op^*(\cc{E},\,\cc{F})
                                \ar[r]^{\rho}
        & pc\cc{S}it^*(\cc{C},\,\cc{F})
       }
$$
In the diagram on the right the arrows $\rho_\varepsilon$,
$\rho_\lambda$ and $\rho$ are equivalences of categories
(\ref{Diaconescu}, \ref{sitelimit} and \ref{pseudodiaconescu}
respectively), so it follows that $\rho_l$ is an equivalence. This
finishes the proof.
\end{proof}
This theorem shows the existence of small 2-cofiltered bilimits in the
\mbox{2-category} of topoi and geometric morphisms. But, it shows more,
namely, that given any small 2-filtered diagram of topoi, without loss of
generality, we can construct a
small site with finite limits for the bilimit topos out of a 2-cofiltered sub-diagram
of small sites with finite limits. However, this depends on the
\emph{axiom of choice} (needed for Proposition \ref{res}). We notice
for the interested reader that if we allow large sites (as in Theorem
\ref{sitelimit}), we can take the topoi themselves as sites, and the
proof of theorem \ref{main} with $\cc{C} = \cc{E}$ does not use
Proposition \ref{res}. Thus, without the use of choice we have: 
\begin{theorem} \label{main2}
	Let $\A^{op}$ be a small 2-filtered 2-category, and $\A^{op}
        \xto \E \cc{T}op$ be a 2-functor. Then, the topos of sheaves $\widetilde{\Colim{\cc{E}}}$ on
the site $\Colim{\cc{E}}$ of \ref{sitelimit} is a bilimit of 
$\cc{E}$ in $\cc{T}op$, or, equivalently, a bicolimit in
$\cc{T}op^*$.  
\end{theorem}


\begin{thebibliography}{00}

\bibitem{G1} Artin M, Grothendieck A, Verdier J.,  \textsl{SGA 4 ,
(1963-64)}, Lecture Notes in Mathematics 269  Springer, (1972).

\bibitem{G2} Artin M, Grothendieck A, Verdier J.,  \textsl{SGA 4 ,
(1963-64)}, Springer Lecture Notes in Mathematics  270 (1972).

\bibitem{D} Dubuc, E. J., \textsl{2-Filteredness and the point of
  every Galois topos}, Proceedings of CT2007, Applied Categorical
  Structures, Volume 18, Issue 2, Springer Verlag (2010).

\bibitem{DS} Dubuc, E. J., Street, R., \textsl{A construction of
  2-filtered bicolimits of categories}, Cahiers de Topologie et
  Geometrie Differentielle, (2005).

\bibitem{G} Gray J. W., \textsl{Formal Category Theory: Adjointness
  for $2$-Categories}, Springer Lecture Notes in Mathematics 391
  (1974).

\bibitem{MM}  Mac Lane S., Moerdijk I., \textsl{Sheaves in Geometry 
and Logic}, Springer Verlag, (1992).

\bibitem{S} Street R.,\textsl{Limits indexed by category-valued $2$-functors} 
 J. Pure Appl. Alg. 8 (1976).

\end{thebibliography}
\end{document}